\documentclass{article}
\usepackage{amsthm, amsfonts, amssymb,latexsym}

\title{On Products of Shifts in Arbitrary Fields}
\newcommand{\F}{\mathbb{F}}
\newcommand{\Z}{\mathbb{Z}}
\newcommand{\R}{\mathbb{R}}
\usepackage{amsmath}
\usepackage[a4paper, total={6in, 8in}]{geometry}
\author{Audie Warren}

\newtheorem{lem}{Lemma}
\newtheorem{cor}{Corollary}

\newtheorem{thm}{Theorem}
\newtheorem{conj}{Conjecture}
\begin{document}
 \maketitle
 
 \begin{abstract}
     We adapt the approach of Rudnev, Shakan, and Shkredov presented in \cite{ShShRu} to prove that in an arbitrary field $\F$, for all $A \subset \F$ finite with $|A| < p^{1/4}$ if $p:= Char(\F)$ is positive, we have
$$|A(A+1)| \gtrsim |A|^{11/9}, \qquad |AA| + |(A+1)(A+1)| \gtrsim |A|^{11/9}.$$
This improves upon the exponent of $6/5$ given by an incidence theorem of Stevens and de Zeeuw.
 \end{abstract}
 \section{Introduction and Main Result}
For finite $A \subseteq \F$, we define the \textit{sumset} and \textit{product set} of $A$ as 
$$A+A = \{a + b : a,b \in A \}, \quad  AA = \{ ab : a,b\in A \}.$$ It is an active area of research to show that one of these sets must be large relative to $A$. The central conjecture in this area is the following.
\begin{conj}[Erd\H{o}s - Szemer\'{e}di]\label{ESconj}
For all $\epsilon > 0$, and for all $A \subseteq \Z$ finite, we have $$|AA| + |A+A| \gg |A|^{2-\epsilon}.$$
\end{conj} The notation $A \ll B$ is used to hide absolute constants, and in addition the notation $A \lesssim B$ is used to hide constant factors and factors of $\log|A|$, i.e. $A \lesssim B$ if and only if there exist absolute constants $c >0$ and $d$ such that $A \leq c B(\log|A|)^d$. If $A \lesssim B$ and $B \lesssim A$ we write $A \sim B$.
Although Conjecture \ref{ESconj} is stated over the integers it can be considered over fields, the real numbers being of primary interest. Current progress over $\R$ places us at an exponent of $\frac{4}{3} + c$ for some small $c$, due to Shakan \cite{Shak}, building on works of Konyagin and Shkredov \cite{KonShk} and Solymosi \cite{Soly}. Incidence geometry, and in particular the Szemer\'{e}di-Trotter Theorem, are the tools used to prove such results in the real numbers.

Conjecture \ref{ESconj} can also be considered over arbitrary fields $\F$. We will let $p$ denote the characteristic of $\F$ throughout. Due to the possible existence of subfields in $\F$, extra restrictions on $|A|$ relative to $p$ must be imposed if $p >0$. All such conditions can be ignored if $p=0$. Over arbitrary fields we replace the Szemer\'{e}di-Trotter Theorem with a point-plane incidence theorem of Rudnev \cite{RuPoPl}, which was used by Stevens and de Zeeuw to derive a point-line incidence theorem \cite{SDZ}. The exponent of $6/5$ was proved in 2014 by Roche-Newton, Rudnev, and Shkredov \cite{ORS}. An application of the Stevens - de Zeeuw Theorem also gives this exponent of $6/5$ for Conjecture \ref{ESconj}, so that $6/5$ became a threshold to be broken.

The $6/5$ threshold has recently been broken, see \cite{ShSh}, \cite{ShShRu}, and \cite{ChKeMo}. The following theorem was proved in \cite{ShShRu} by Rudnev, Shakan, and Shkredov, and is the current state of the art bound.
\begin{thm} {\normalfont \cite{ShShRu}} \label{sp}
Let $A \subset \F$ be finite with $|A| < p^{18/35}$. Then
$$ |A+A| + |AA| \gtrsim |A|^{11/9}.$$
\end{thm} 
 
Another way of considering the sum-product phenomenon is to consider the set $A(A+1)$, which we would expect to be quadratic in size. This encapsulates the idea that a translation of a multiplicatively structured set should destroy its structure, which is a main theme in sum-product questions. Study of growth of $|A(A+1)|$ began in \cite{GS} by Garaev and Shen, see also \cite{OlTJ}, \cite{DZ}, and \cite{AM}. Current progress for $|A(A+1)|$ comes from an application of the Stevens - de Zeeuw Theorem, giving the same exponent of $6/5$. In this paper we use the multiplicative analogue of ideas in \cite{ShShRu} to prove the following theorem. 
\begin{thm}\label{shift}
Let $A, B, C, D \subset \F$ be finite with the conditions $$|C(A+1)||A| \leq |C|^3, \ \ |C(A+1)|^2 \leq |A||C|^3, \ \ |B| \leq |D|, \ \  |A|, |B|, |C|,|D| < p^{1/4}.$$ Then we have
$$ |AB|^8 |C(A+1)|^2 |D(B - 1)|^8 \gtrsim |B|^{13}|A|^5 |C|^3 |D|.$$
\end{thm}
In our applications of this theorem we have $|A| = |B| = |C| = |D|$, so that the first three conditions are trivially satisfied. The conditions involving $p$ could likely be improved, however for sake of exposition we do not attempt to optimise these. The main proof closely follows \cite{ShShRu} (in the multiplicative setting), the central difference being a bound on multiplicative energies in terms of products of shifts. An application of Theorem \ref{shift} beats the threshold of $6/5$, matching the $11/9$ appearing in Theorem \ref{sp}. Specifically, we have
\begin{cor} \label{A(A+1)}
Let $A \subseteq \F$ be finite, with $|A| < p^{1/4}$. Then 
$$|A(A+1)| \gtrsim |A|^{11/9}, \qquad |AA| + |(A+1)(A+1)| \gtrsim |A|^{11/9}.$$
\end{cor}

 \subsection{Acknowledgements}
 The author was supported by the Austrian Science Fund (FWF) Project P 30405-N32. The author would also like to thank Oliver Roche-Newton and Misha Rudnev for helpful conversations.
 \section{Preliminary Results}
 We require some preliminary theorems. The first is the point-line incidence theorem of Stevens and de Zeeuw.
 
 \begin{thm}[Stevens - de Zeeuw, \cite{SDZ}] \label{SD}
 Let $A$ and $B$ with $|A| \geq |B|$ be finite subsets of $\F$ a field, and let $L$ be a set of lines. Assuming $|L||B| \ll p^2$ and $|B||A|^2 \leq |L|^3$, we have
 $$I(A \times B, L) \ll |A|^{1/2}|B|^{3/4}|L|^{3/4} + |L|$$
 \end{thm}
 
Note that as $A$ is the larger of $A$ and $B$, we may swap the powers of $|A|$ and $|B|$ in this result. Before stating the next two theorems we require some definitions. For $x \in \F$ we define the \emph{representation function} 
$$r_{A/D}(x) = \left| \left\{ (a,d) \in A \times D : \frac{a}{d} = x \right\} \right|.$$
The set $A/D$ in this definition can be changed to any other combination of sets, changing the fraction $\frac{a}{d}$ in the definition to match. We also define the $n$'th moment \emph{multiplicative energy} of sets $A$, $D \subseteq \F$ as 
$$E_n^*(A,D) = \sum_x r_{A/D}(x)^n.$$

We use Theorem \ref{SD} to prove two further results. The first is a bound on the fourth order multiplicative energy relative to products of shifts.
 
 \begin{thm} \label{E4} For all finite non-empty $A,C,D \subset \F$ with $|A|^2|C(A+1)| \leq |D||C|^3$, $|A||C(A+1)|^2 \leq |D|^2|C|^3$, and $|A||C||D|^2 \ll p^2$, we have
 $$E_4^*(A,D) \lesssim \min \left\{  \frac{|C(A+1)|^2|D|^3}{|C|}, \frac{|C(A+1)|^3|D|^2}{|C|} \right\} $$
 \end{thm}
 The second result is similar, but for the second moment multiplicative energy.
 \begin{thm}\label{e2bound}
 For all finite and non-empty $A$, $C$, $D \subset \F$  with $|A|^2|C(A+1)| \leq |D||C|^3$,  $|A||C(A+1)|^2 \leq |D|^2|C|^3$, and $|A||C||D|\min\{|C|,|D|\} \ll p^2$, we have
 $$E^*(A,D) \lesssim \frac{|C(A+1)|^{3/2}|D|^{3/2}}{|C|^{1/2}}.$$

 \end{thm}
 The $A+1$ appearing in these theorems can be changed to any $A+
\lambda$ for $\lambda \neq 0$, by noting that $|C(A+1)| = |C(\lambda A + \lambda)|$ and renaming $A' = \lambda A$. For our purposes, we will use $\lambda = \pm 1$. 
 \begin{proof}[Proof of Theorem \ref{E4}]
WLOG we can assume that $0 \notin A,C,D$. We begin by proving that 
 $$E_4^*(A,D) \lesssim  \frac{|C(A+1)|^2|D|^3}{|C|}. $$
 Define the set
 $$S_{\tau} := \{ x \in A/D :\tau \leq r_{A/D}(x) < 2\tau \}.$$
 By a dyadic decomposition, there is some $\tau$ with 
 $$|S_{\tau}|\tau^4 \sim E_4^*(A,D).$$
 Take an element $t \in S_{\tau}$. It has $\tau$ representations in $A/D$, so there are $\tau$ ways to write $t = a/d$ with $a \in A$, $d \in D$. For all $c \in C$, we have
 \begin{align*}
     t &= \frac{a}{d} \\
     & = \frac{1}{d}\left(\frac{ac+c-c}{c}
\right) \\
& = \frac{1}{d} \left( \frac{p}{c} - 1\right)\end{align*}
where $p = c(a+1) \in C(A+1)$. This shows that we have $|S_{\tau}|\tau|C|$ incidences between the lines 
$$L = \{ l_{d,c}: d\in D, c \in C\}, \quad l_{d,c} \text{  given by  } y = \frac{1}{d}\left(\frac{x}{c}-1\right)$$
and the point set $P = C(A+1) \times S_{\tau}$. Under the conditions $|D||C|\min\{|S_{\tau}|,|C(A+1)| \} \ll p^2$ and $|S_{\tau}||C(A+1)| \max \{ |S_{\tau}|,|C(A+1)| \} \leq |D|^3|C|^3$, we have that 
$$|S_{\tau}|\tau |C| \leq I(P,L) \ll |C(A+1)|^{1/2}|S_{\tau}|^{3/4}|C|^{3/4}|D|^{3/4} + |D||C|.$$
The conditions are satisfied under the assumptions $|D||A||C|\min\{|D|,|C|\} \ll p^2$, $|A|^2|C(A+1)| \leq |D||C|^3$, and $|A||C(A+1)|^2 \leq |D|^2|C|^3$ . Assuming that the leading term is dominant, we have
$$|S_{\tau}|\tau^4 |C| \ll |C(A+1)|^{2}|D|^{3}$$
so that as $|S_{\tau}|\tau^4 \sim E_4^*(A,D)$, we have
$$E_4^*(A,D) \lesssim \frac{|C(A+1)|^2|D|^3}{|C|}.$$
We therefore assume the leading term is not dominant. Suppose $|D||C|$ is dominant, so that
\begin{equation}
     \label{1} |D||C| \gg |C(A+1)|^{1/2}|S_{\tau}|^{3/4}|C|^{3/4}|D|^{3/4}.\end{equation}
Raising to the power four and multiplying through by $\tau^{12}$ we get the bound
$$|D||C| \tau^{12} \gtrsim |C(A+1)|^2 E_4^*(A,D)^3 \implies  E_4^*(A,D) \lesssim \frac{|D|^{1/3}|C|^{1/3}\tau^4}{|C(A+1)|^{2/3}}.$$
We now assume that the result doesn't hold, that is
$$\frac{|C(A+1)|^2|D|^3}{|C|} \lesssim \frac{|D|^{1/3}|C|^{1/3}\tau^4}{|C(A+1)|^{2/3}}$$
which gives
$$ |D|^8|C|^4|A|^4\ll|D|^8|C(A+1)|^8 \lesssim \tau^{12} |C|^4 \ll |A|^{12} |C|^4$$
so that we have $|D| \lesssim |A|$. We return to equation \eqref{1} and simplify, to find 
\begin{align*}
    |A|^{1/4}|C|^{1/4} \gtrsim |D|^{1/4}|C|^{1/4} & \gg  |C(A+1)|^{1/2}|S_{\tau}|^{3/4} \geq |A|^{1/4}|C|^{1/4}|S_{\tau}|^{3/4}  
\end{align*}
so that $|S_{\tau}| \sim 1$. We then have
$$ |D||C| \gg |C(A+1)|^{1/2}|S_{\tau}|^{3/4}|C|^{3/4}|D|^{3/4} \sim  |C(A+1)|^{1/2}|C|^{3/4}|D|^{3/4} \gg |A|^{1/4}|C||D|^{3/4} \gtrsim |D||C|$$
so that the two terms are in fact balanced and the result follows.

Secondly, we prove that 
$$E_4^*(A,D) \lesssim  \frac{|C(A+1)|^3|D|^2}{|C|}. $$
To do this, we swap the roles of $D$ and $S_{\tau}$ from above. We define the line set and point set by
$$L = \{ l_{t,c}: t \in S_{\tau}, c \in C \}, \qquad P = C(A+1) \times D.$$
Any incidence from the previous point and line set remains an incidence for the new ones, via $t = \frac{1}{d}\left(\frac{p}{c}-1\right) \iff d = \frac{1}{t}\left(\frac{p}{c}-1\right)$. Under the conditions
\begin{equation} \label{condterm2}
    |S_{\tau}||C| \min\{|D|,|C(A+1)|\} \ll p^2, \qquad |D||C(A+1)|\max\{|D|, |C(A+1)| \} \leq |S_{\tau}|^3|C|^3
\end{equation}
we have
$$|S_{\tau}|\tau |C| \ll I(P,L) \ll |C(A+1)|^{3/4}|S_{\tau}|^{3/4}|C|^{3/4}|D|^{1/2} + |S_{\tau}||C|.$$
If the leading term dominates, the result follows from $|S_{\tau}|\tau^4 \sim E_4^*(A,D)
$. Assume the leading term is not dominant, that is,
$$|C(A+1)|^{3}|D|^{2} \ll |S_{\tau}||C|.$$ Then by using $|S_{\tau}| \ll |A||D|$, we have
$$|A||C|^2|D|^2 \leq |C(A+1)|^3 |D|^2 \ll |A||D||C|$$ so that $|C| \sim |D| \sim 1$ and the result is trivial. 

 We now check the conditions \eqref{condterm2} for using Theorem \ref{SD}. The first condition in \eqref{condterm2} is satisfied if $|A||C||D|^2 \ll p^2$, which is true under our assumptions. The second condition depends on $\max \{ |D|, |C(A+1)| \}$, which we assume is $|D|$ (if not the first term in Theorem \ref{E4} gives stronger information, which we have already proved). Assuming the second condition does not hold, we have
 $$|D|^2 |C(A+1)| \gg |S_{\tau}|^3 |C|^3.$$
 Multiplying by $\tau^{12}$ on both sides and bounding $\tau \ll |A|$, we get
 \begin{equation} \label{e4boundc2}
     E_4^*(A,D) \lesssim \frac{|A|^4|D|^{2/3}|C(A+1)|^{1/3}}{|C|}.
 \end{equation} 
 Assuming that the result does not hold, we have 
 $$\frac{|C(A+1)|^3|D|^2}{|C|} \lesssim  \frac{|A|^4|D|^{2/3}|C(A+1)|^{1/3}}{|C|}$$
 giving
 $$ |A|^8|D|^4 \ll |C(A+1)|^{8}|D|^{4} \lesssim |A|^{12}. $$
 So that $|D| \lesssim |A|$. In turn, this implies $|A| \gtrsim |D| \geq |C(A+1)| \gg |A|$, so that $|A| \sim |C(A+1)| \sim |D|$. Returning to equation \ref{e4boundc2}, this gives
 $$  E_4^*(A,D) \lesssim \frac{|A|^4|D|^{2/3}|C(A+1)|^{1/3}}{|C|} \sim \frac{|C(A+1)|^3|D|^2}{|C|}$$
 and the result is proved.
 \end{proof}
 
 \begin{proof}[Proof of Theorem \ref{e2bound}]
The proof follows similarly to that of Theorem \ref{E4}. We again define the lines and points
$$L = \{ l_{d,c}: d\in D, c \in C\}, \quad l_{d,c} \text{  given by  } y = \frac{1}{d}\left(\frac{x}{c}-1\right), \quad P = C(A+1) \times S_{\tau},$$
where in this case the set $S_{\tau}$ is rich with respect to $E^*(A,D)$, so that $|S_{\tau}|\tau^2 \sim E^*(A,D)$. With the conditions $|A||C||D|\min\{|D|,|C|\} \ll p^2$ and $|S_{\tau}||C(A+1)| \max \{ |S_{\tau}|,|C(A+1)| \} \leq |D|^3|C|^3$, (which are satisfied under our assumptions) we have by Theorem \ref{SD},
$$|S_{\tau}| \tau |C| \ll I(P,L) \ll |S_{\tau}|^{1/2}|C(A+1)|^{3/4}|D|^{3/4}|C|^{3/4} + |D||C|.$$
 If the leading term dominates, we have
 $$|S_{\tau}|\tau^2 \ll \frac{|C(A+1)|^{3/2} |D|^{3/2}}{|C|^{1/2}} $$
 and the result follows from $|S_{\tau}|\tau^2 \sim E^*(A,D)$. We therefore assume that the leading term does not dominate, that is,
 $$ |S_{\tau}|^{1/2}|C(A+1)|^{3/4}|D|^{3/4}|C|^{3/4} \ll |D||C|. $$
 Multiplying through by $\tau$ and squaring, we get the bound
 $$ E^*(A,D) \lesssim \frac{|D|^{1/2}|C|^{1/2}\tau^2}{|C(A+1)|^{3/2}}.$$
 Assuming the result does not hold, we have
 $$\frac{|D|^{3/2}|C(A+1)|^{3/2}}{|C|^{1/2}} \lesssim \frac{|D|^{1/2}|C|^{1/2}\tau^2}{|C(A+1)|^{3/2}} \implies |D||C(A+1)|^3 \lesssim |C| \tau^2.$$
 Bounding $\tau \leq |A|$ and $|C||A|^2 \ll |C(A+1)|^3$ we have $|D| \sim 1$. Similarly, bounding $\tau^2 \leq |A||D|$ and $|C(A+1)|^3 \geq |C|^2|A|$, we find $|C| \lesssim 1$, so that the result is trivial.
\end{proof}

\section{Proof of Theorem \ref{shift}}

 We follow a multiplicative analogue of the argument in \cite{ShShRu}. For $A$ and $B$ finite subsets of $\F$, define a popular set of products
$$P := \left\{ x \in AB : r_{AB}(x) \geq \frac{|A||B|}{\log |A||AB|} \right\}.$$
Note that by writing
$$| \{ (a,b) \in A \times B : ab \in P \} | + | \{ (a,b) \in A \times B : ab \notin P \}  | = |A||B| $$
and noting that $$| \{ (a,b) \in A \times B : ab \notin P \}|  = \sum_{x \notin P} r_{AB}(x) < |P^c|\frac{|A||B|}{\log |A||AB|} \leq \frac{|A||B|}{\log|A|}$$
we have $$| \{ (a,b) \in A \times B : ab \in P \} | \geq  \left( 1-\frac{1}{\log|A|} \right) |A||B|.$$ We also define a popular subset of $A$ with respect to $P$, as
$$A' := \left\{ a \in A : |\{ b \in B : ab \in P \}| \geq \frac{2}{3} |B| \right\}.$$
We have
$$ | \{ (a,b) \in A \times B : ab \in P \} | = \sum_{a \in A'}| \{ b : ab \in P \}| + \sum_{a \in A \setminus A'}| \{ b : ab \in P \} |  \geq  \left( 1-\frac{1}{\log|A|} \right)|A||B|$$
Suppose that $|A \setminus A'| = c |A|$ for some $c\geq0$, so that also $|A'| = (1-c)|A|$. Noting that 
$$\sum_{a \in A'}| \{ b : ab \in P \}|  \leq (1-c) |A| |B|, \qquad \sum_{a \in A \setminus A'}| \{ b : ab \in P \} | \leq \frac{2c}{3}|A||B|, $$ we have 
$$ (1-c) |A| |B| + \frac{2c}{3}|A||B| \geq (1-\frac{1}{\log|A|})|A||B|  \implies c < \frac{3}{\log|A|},$$
so that $|A'| \geq \left(1 - \frac{3}{\log |A|}  \right)|A|$.

We use a multiplicative version of Lemma 8 in \cite{ShShRu}. The proof we present is an expanded version of the proof present in \cite{ShShRu}. 

\begin{lem} \label{4/3decomp}
For all $A \subset \F$, there exists $A_1 \subseteq A$ with $|A_1| \gg |A|$, such that 
$$E^*_{4/3}(A_1') \gg E^*_{4/3}(A_1) $$ 
\end{lem}
\begin{proof}
We give an algorithm which shows such a subset exists, as otherwise we have a contradiction. We recursively define 
$$A_i = A_{i-1}', \qquad A_0 = A, \qquad i < \log |A|$$
where $A_i'$ is defined relative to $A_i$. Using the same arguments as above, we have
$|A_i'| \geq \left( 1- \frac{3}{\log|A|} \right)|A_i|$, so that $|A_i| \gg |A|$ for all $i$, by following the chain $|A_i| = |A_{i-1}'| \geq\left( 1- \frac{3}{\log|A|} \right)|A_{i-1}| \geq ... \geq \left( 1- \frac{3}{\log|A|} \right)^{\log|A|} |A_0| \geq \frac{|A|}{e^3}$. We assume that at all steps, we have
$$E_{4/3}^*(A_i') < \frac{E^*_{4/3}(A_i)}{4} $$
as otherwise we have $E_{4/3}^*(A_i') \gg E^*_{4/3}(A_i) $ and we are done. After $\log|A|$ steps, we have a set $A_k$ with
$$|A_k| \gg |A|, \qquad E_{4/3}^*(A_k') < \frac{E^*_{4/3}(k)}{4} < \frac{E^*_{4/3}(A_{k-1})}{16} < ... < \frac{E^*_{4/3}(A)}{4^{\log|A|}}.$$
But then we have 
$$E^*_{4/3}(A) > E_{4/3}^*(A_k') 4^{\log |A|} \gg |A|^{4/3 + 2} = |A|^{10/3}$$
which is a contradiction. Therefore at some step we have an $A_i$ satisfying the lemma.\end{proof}

We apply this lemma at the outset, redefining the subset $A_i$ found by Lemma \ref{4/3decomp} as $A$ to ensure WLOG that we have
$$E^*_{4/3}(A') \gg E^*_{4/3}(A). $$ 
We pigeonhole the ratio set $A'/A'$ in relation to the energy $E_{4/3}^*(A')$ to find a set $Q \subset A'/A'$ with $|Q|\Delta^{4/3} \sim E_{4/3}^*(A')$ for some $\Delta > 0$.

We will bound the number of solutions to the trivial equation
\begin{equation} \label{treq}
    \frac{a}{a'} = \frac{ab}{a'b} = \frac{ab'}{a'b'}
\end{equation}
such that $(a,a',b,b') \in A^2 \times B^2$, $\frac{a}{a'} \in Q$, $ab,ab',a'b,a'b' \in P$. We have
$$N = \sum_{\substack{a,a' \in A' \\ a/a' \in Q}} | \{ b \in B : ab, a'b \in P \} |^2 $$
and we see that as for all $a \in A'$, $| \{ b \in B : ab \in P \} | \geq \frac{2}{3}|B|$, by considering the intersection of $\{ b \in B : ab \in P \}$ and $\{ b \in B : a'b \in P \}$, we have that for all $a,a' \in A'$, $| \{ b \in B : ab, a'b \in P \} | \geq \frac{1}{3}|B|$. Therefore $N \geq \frac{1}{9}|B|^2|Q|\Delta$.

Define an equivalence relation on $A^2 \times B^2$ via 
$$(a,a',b,b') \sim (c,c',d,d') \iff \exists \ \lambda \text{ s.t.} \ a = \lambda c, a' = \lambda c', b = \frac{d}{\lambda}, b' = \frac{d'}{\lambda}. $$
Note that the conditions $\frac{a}{a'} \in Q$, $ab,a'b,ab',a'b' \in P$ are invariant in the class (i.e. if one class element satisfies these conditions, then they all do). Call the number of equivalence classes satisfying these conditions $|X|$. Also note that any quadruple satisfying these conditions gives a solution to \eqref{treq}. We can therefore write the number of solutions $N$ as the sum over each equivalence class;
$$N = \sum_{\substack{[a,a',b,b'] \\ ab,a'b,ab',a'b' \in P \\ \frac{a}{a'} \in Q}} \left|[a,a',b,b']\right|.$$
 By Cauchy-Schwarz and completing the sum over all equivalence classes, we have
$$|Q|^2\Delta^2|B|^4 \ll N^2 \leq |X|  \sum_{[a,a',b,b']} \left|[a,a',b,b']\right|^2$$
We must now bound the two quantities on the right hand side of this equation. We first claim that 
\begin{equation} \label{equivenergy} \sum_{[a,a',b,b']} \left|[a,a',b,b']\right|^2 \leq \sum_x r_{A/A}(x)^2r_{B/B}(x)^2.\end{equation}
To see this, note that the left hand side counts pairs of elements of equivalence classes. Take any two elements $(a,a',b,b'), \ (c,c',d,d') \in A^2 \times B^2$ from the same equivalence class, so that we may write $(c,c',d,d')= (\lambda a,\lambda a', \frac{b}{\lambda},\frac{b'}{\lambda}) $. The $8$-tuple $(a,a',b,b',c,c',d,d')$ satisfies
$$\lambda = \frac{c}{a} = \frac{c'}{a'} = \frac{b}{d} = \frac{b'}{d'}$$
and thus contributes to the sum $\sum_x r_{A/A}(x)^2r_{B/B}(x)^2$. We also see that different pairs from equivalence classes give different $8$-tuples, and so the claim is proved. We use Cauchy-Schwarz on the right hand side of equation \ref{equivenergy} to transform it into a pair of fourth energies.
$$ \sum_x r_{A/A}(x)^2r_{B/B}(x)^2 \leq E_4^*(A)^{1/2}E_4^*(B)^{1/2}.$$
We use Theorem \ref{E4} to bound these energies. We bound via
$$E_4^*(A) \lesssim \frac{|C(A+1)|^2|A|^3}{|C|}, \quad E_4^*(B) \lesssim \frac{|D(B-1)|^2|B|^3}{|D|}$$ with conditions 
$$|C(A+1)||A| \leq |C|^3, \ |C(A+1)|^2 \leq |A||C|^3, \ |A|^3|C| \ll p^2$$$$ |D(B-1)||B| \leq |D|^3 ,\ |D(B-1)|^2 \leq |B||D|^3,\  |B|^3|D| \ll p^2$$ which are all satisfied under our assumptions. This gives us
$$ |Q|^2\Delta^2|B|^4 \lesssim  |X|  \frac{|C(A+1)||A|^{3/2}|D(B-1)||B|^{3/2}}{|C|^{1/2}|D|^{1/2}}.$$
We now bound $|X|$, the number of equivalence classes. Note that any $(a,a',b,b')$ a solution to equation \eqref{treq} with the relevant conditions as above transforms into a solution to the equation 
\begin{equation} \label{reducedeq}
    w = \frac{s}{t} = \frac{u}{v}
\end{equation}
with $w \in Q$, $s,t,u,v \in P$, by taking $w = \frac{a}{a'}$, $s = ab,t=a'b,u=ab', v=a'b'$. Note that taking two solutions $(a,a',b,b')$ and $(c,c',d,d')$ that are \emph{not} from the same equivalence class necessarily gives us two different solutions to equation \eqref{reducedeq} via the map above. Thus we have
\begin{align*}
    |X| & \leq \left| \left\{ (w,s,t,u,v) \in Q \times P^4 : w = \frac{s}{t} = \frac{u}{v} \right\} \right| \\
    & = \left| \left\{ (s,t,u,v) \in  P^4 : \frac{s}{t} = \frac{u}{v} \in Q \right\} \right|.
\end{align*} 
The popularity of $P$ allows us to bound this by
$$|X| \lesssim \frac{|AB|^4}{|A|^4|B|^4} \left| \left\{ (a_1,a_2,a_3,a_4,b_1,b_2,b_3,b_4) \in  A^4 \times B^4 : \frac{a_1b_1}{a_2b_2} = \frac{a_3b_3}{a_4b_4} \in Q \right\} \right|. $$ 
We dyadically pigeonhole the set $\frac{BA}{A}$ in relation to the number of solutions to $r/a = r'/a' \in Q$ with $r,r' \in \frac{BA}{A}$, $a,a' \in A$ to find popular subsets $R_1, R_2 \subseteq \frac{BA}{A}$ in terms of these solutions. Specifically, we have
\begin{align*}
    |X| &\lesssim \frac{|AB|^4}{|A|^4|B|^4} \sum_{i} \sum_{\substack{x \in \frac{AB}{A} \\ 2^i \leq r_{\frac{AB}{A}}(x) < 2^{i+1}}} r_{
    \frac{AB}{A}}(x)\left| \left\{ (a_3,a_4,b_1,b_3,b_4) \in  A^2 \times B^3 : \frac{x}{b_1} = \frac{a_3b_3}{a_4a_4} \in Q \right\} \right|
\end{align*}
to give us a $\Delta_1 > 0$ and an $R_1 \subseteq \frac{AB}{A}$ such that 
\begin{align*}
    |X| &\lesssim \frac{|AB|^4}{|A|^4|B|^4} \Delta_1 \left| \left\{ (r_1,a_3,a_4,b_2,b_3,b_4) \in  R_1 \times A^2 \times B^3 : \frac{r_1}{b_2} = \frac{a_3b_3}{a_4b_4} \in Q \right\} \right|.
\end{align*}
We perform a similar dyadic decomposition to get a $\Delta_1'>0$ and $R_2 \subseteq \frac{AB}{A}$ with 
\begin{align*}
    |X| &\lesssim \frac{|AB|^4}{|A|^4|B|^4} \Delta_1 \Delta_1' \left| \left\{ (r_1,r_2,b_2,b_4) \in  R_1 \times R_2 \times B^2 : \frac{r_1}{b_2} = \frac{r_2}{b_4} \in Q \right\} \right|.
\end{align*}
We use these decompositions to get the bound
\begin{align*}
    |X| & \lesssim   \frac{|AB|^4}{|A|^4|B|^4} \Delta_1 \Delta_1' \left| \left\{ (r_1,r_2,b_2,b_4) \in  R_1 \times R_2 \times B^2 : \frac{r_1}{b_2} = \frac{r_2}{b_4} \in Q \right\} \right| \\
     & \leq  \frac{|AB|^4}{|A|^4|B|^4} \Delta_1 \Delta_1' \sum_{q \in Q}r_{R_1/B}(q)r_{R_2/B}(q) \\
     &\leq  \frac{|AB|^4}{|A|^4|B|^4} \Delta_1 \Delta_1' \left( \sum_{q \in Q}r_{R_1/B}(q)^{2} \right)^{1/2} \left( \sum_{q \in Q}r_{R_2/B}(q)^{2} \right)^{1/2} \\
     & \leq \frac{|AB|^4}{|A|^4|B|^4} \Delta_1 \Delta_1' |Q|^{1/2} E_4^*(B, R_1)^{1/4}E_4^*(B, R_2)^{1/4}
\end{align*}
We will now show that given $|B||D||R_i|^2 \ll p^2$ and $|B| \leq |D|$ (which are true under our assumptions), we have
\begin{equation} \label{BRi}
    E_4^*(B,R_i) \lesssim \frac{|D(B-1)|^3|R_i|^2}{|D|}.
\end{equation} 
Firstly, with the additional conditions \begin{equation} \label{condE4}
    |B|^2|D(B-1)| \leq |R_i||D|^3, \quad |B||D(B-1)|^2 \leq |R_i|^2|D|^3
\end{equation} we may bound these fourth energies by Theorem \ref{E4} to get \eqref{BRi}. We can therefore assume one of these conditions does not hold.

Suppose that $ |B|^2 |D(B-1)| \geq |R_i||D|^3$. We have
$$
    E_4^*(B,R_i) \leq |R_i|^4|B|.
$$
Note that we want to have
$$E_4^*(B,R_i) \leq \frac{|D(B-1)|^3|R_i|^2}{|D|}$$ 
which would follow from
$$|R_i|^4 |B| \leq \frac{|D(B-1)|^3|R_i|^2}{|D|}$$
which is true if and only if $|R_i|^2 |B||D| \leq |D(B-1)|^3$. Using our assumed bound for $|R_i|$, we know that
$$|R_i|^2 |B||D| \leq \frac{|B|^5|D(B-1)|^2}{|D|^5}. $$ Noting that we certainly have 
$$|B| \leq |D| \implies \frac{|B|^5|D(B-1)|^2}{|D|^5} \leq |D(B-1)|^3$$ so that we must have
$$|R_i|^2 |B||D| \leq |D(B-1)|^3 $$ and so the bound on the fourth energy holds. 

Now assume the second condition from \eqref{condE4} does not hold, that is, $|B||D(B-1)|^2 \geq |R_i|^2|D|^3$. Again, we have
\begin{align*}
    E_4^*(B,R_i) &\leq |R_i|^4|B|. 
\end{align*}
We have 
$$ |R_i|^4|B| \leq \frac{|D(B-1)|^3|R_i|^2}{|D|} \iff |R_i|^2|B||D| \leq |D(B-1)|^3$$
so that it is enough to prove $|R_i|^2|B||D| \leq |D(B-1)|^3$, as before. Using the assumption to bound $|R_i|$, we have the information that
$$|R_i|^2 |B||D| \leq \frac{|B|^2|D(B-1)|^2}{|D|^2}$$
and it follows from our assumption $|B| \leq |D|$ that
$$ \frac{|B|^2|D(B-1)|^2}{|D|^2} \leq |D(B-1)|^3.$$
Therefore we have that $|R_i|^2|B||D| \leq |D(B-1)|^3$ and so the bound on the fourth energy holds. Plugging this in, we get
\begin{align*}
  |X| & \lesssim \frac{|AB|^4}{|A|^4|B|^4} \Delta_1 \Delta_1' |Q|^{1/2} E_4^*(B, R_1)^{1/4}E_4^*(B, R_2)^{1/4} \\
  & \lesssim \frac{ |Q|^{1/2} |AB|^4|D(B-1)|^{3/2}}{|A|^4|B|^4|D|^{1/2}}|R_1|^{1/2}|R_2|^{1/2} \Delta_1 \Delta_1'
\end{align*}
The product $|R_1|^{1/2}|R_2|^{1/2}\Delta_1 \Delta_1'$ can be bounded by
$$\left[|R_1|^{1/2}|R_2|^{1/2}\Delta_1 \Delta_1'\right]^2 \leq \sum_{x\in R_1}r_{\frac{BA}{A}}(x)^2 \sum_{x\in R_2}r_{\frac{BA}{A}}(x)^2 $$
where it is important to note that $r_{\frac{BA}{A}}(x)$ gives a triple $(b,a,a')$. For $i =1,2$, we have
\begin{align*}
\sum_{x\in R_i}r_{\frac{BA}{A}}(x)^2   &\leq \left| \left\{ (a_1,a_2,a_3,a_4,b_1,b_2) \in A^4 \times B^2 : \frac{b_1a_1}{a_2} = \frac{b_2a_3}{a_4} \right\} \right|.
\end{align*}
Following the process as before, we find a pair of subsets $S_1,S_2 \subseteq A/A$ with respect to these solutions, and some $\Delta_2, \Delta_2' >0$ with 
\begin{align*}
   \sum_{x\in R_i}r_{\frac{BA}{A}}(x)^2  
    & \lesssim \Delta_2 \Delta_2'  \left| \left\{ (s_1,s_2,b_1,b_2) \in S_1 \times S_2 \times B^2 : s_1b_1 = s_2b_2 \right\} \right| \\
    & \leq \Delta_2 \Delta_2' \sum_{x} r_{S_1B}(x) r_{S_2B}(x) \\ 
    & \leq \Delta_2 \Delta_2' E^*(B,S_1)^{1/2}E^*(B,S_2)^{1/2}.
\end{align*}
We will use a similar argument as above to prove that with the two conditions $|B||D||S_i|\min\{ |D|, |S_i| \} \ll p^2$ and $B\leq D$ (which are satisfied under our assumptions), we have
\begin{equation} \label{BSi}
    E^*(B,S_i) \lesssim \frac{|S_i|^{3/2}|D(B-1)|^{3/2}}{|D|^{1/2}}.
\end{equation}

Under the extra conditions 
\begin{equation} \label{condE2}
    |B|^2|D(B-1)| \leq |S_i||D|^3,\qquad |B||D(B-1)|^2 \leq |S_i|^2|D|^3
\end{equation} we can bound this energy by Theorem \ref{e2bound} to get \eqref{BSi}. We therefore assume the first condition from \eqref{condE2} does not hold, that is, $|B|^2|D(B-1)| \geq |S_i||D|^3$. We bound the energy via
\begin{align*}
E^*(B,S_i) &\leq |B| |S_i|^2.
\end{align*}
We wish to show that 
$$|B||S_i|^2 \leq \frac{|S_i|^{3/2}|D(B-1)|^{3/2}}{|D|^{1/2}}\ \ \ \text{  which is true iff  \ } \ |B||D|^{1/2}|S_i|^{1/2} \leq |D(B-1)|^{3/2}.$$
Using our assumption on $|S_i|$, we have that
$$|B||D|^{1/2}|S_i|^{1/2} \leq \frac{|B|^2|D(B-1)|^{1/2}}{|D|}. $$
Our assumption that $|B| \leq |D|$ gives
$$ \frac{|B|^2|D(B-1)|^{1/2}}{|D|} \leq |B||D(B-1)|^{1/2} \leq |D(B-1)|^{3/2}$$
so that $|B||D|^{1/2}|S_i|^{1/2} \leq |D(B-1)|^{3/2}$, and the bound \eqref{BSi} holds. Next we assume that the second condition in \eqref{condE2} does not hold, that is, $|B||D(B-1)|^2 \geq |S_i|^2|D|^3 $. We again have
$$E^*(B,S_i) \leq |B||S_i|^2.$$
Comparing this bound to our desired bound, we have
$$|B||S_i|^2 \leq \frac{|S_i|^{3/2}|D(B-1)|^{3/2}}{|D|^{1/2}} \iff |B||D|^{1/2}|S_i|^{1/2} \leq |D(B-1)|^{3/2}.$$
So that the bound we want follows from the second inequality above. Using our assumption on $|S_i|$, we know that
$$|B||D|^{1/2}|S_i|^{1/2} \leq  \frac{|B|^{5/4}|D(B-1)|^{1/2}}{|D|^{1/4}}$$
and by our assumption that $|B| \leq |D|$, we have
$$\frac{|B|^{5/4}|D(B-1)|^{1/2}}{|D|^{1/4}} \leq |D(B-1)|^{3/2}$$
so that we have $|B||D|^{1/2}|S_i|^{1/2} \leq |D(B-1)|^{3/2}$ as needed.

In all cases the bound on $E^*(B,S_i)$ holds, so that we find
\begin{align*}
 \left[|R_1|^{1/2}|R_2|^{1/2}\Delta_1 \Delta_1'\right]^2 & \lesssim \Delta_2^2 \Delta_2'^2 E^*(B,S_1)E^*(B,S_2)\\
   & \lesssim  \frac{\Delta_2^2 \Delta_2'^2|S_1|^{3/2}|S_2|^{3/2}|D(B-1)|^{3}}{|D|} \\
   & \leq \frac{E_{4/3}^*(A)^3|D(B-1)|^3}{|D|}.
\end{align*}

Putting all these bounds together, we have
$$ |Q|^{3/2}\Delta^2|B|^{13/2}|A|^{5/2}|C|^{3/2}|D|^{1/2} \lesssim  |AB|^4|C(A+1)||D(B-1)|^{4}E_{4/3}^*(A)^{3/2} $$
which simplifies to
$$E_{4/3}^*(A')^{3} |B|^{13} |A|^5|C|^3|D| \lesssim |AB|^8 |C(A+1)|^2 |D(B-1)|^8E_{4/3}^*(A)^{3}.$$
We use our assumption $E_{4/3}(A') \gg E_{4/3}(A)$ to conclude that 
$$|B|^{13} |A|^5|C|^3|D| \lesssim |AB|^8 |C(A+1)|^2 |D(B-1)|^8.$$
\qed

The first part of Corollary \ref{A(A+1)} can be seen by setting $B = A+1$, $C=A$ and $D = A+1$ to give
$$|A(A+1)| \gtrsim |A|^{11/9}.$$
Alternatively, setting $B = -A$, $D = C = A+1$ gives the second part,
$$|AA| + |(A+1)(A+1)| \gtrsim |A|^{11/9}.$$

\end{document}